\documentclass[12pt,a4paper]{article}
\usepackage{amssymb,amsmath,amsthm}
\usepackage{graphicx}
\usepackage{color}

\newcommand\cA{{\mathcal A}}

\newcommand\cC{{\mathcal C}}

\newcommand\cF{{\mathcal F}}

\newcommand{\abs}[1]{\left\lvert{#1}\right\rvert}
\newcommand{\floor}[1]{\left\lfloor{#1}\right\rfloor}

\newcommand\cG{{\mathcal G}}

\newcommand\cM{{\mathcal M}}
\newcommand\cN{{\mathcal N}}
\newcommand\cP{{\mathcal P}}

\newcommand\cQ{{\mathcal Q}}

\theoremstyle{plain}
\newtheorem{theorem}{Theorem}[section]
\newtheorem{lemma}[theorem]{Lemma}
\newtheorem{corollary}[theorem]{Corollary}
\newtheorem{conjecture}[theorem]{Conjecture}
\newtheorem{proposition}[theorem]{Proposition}
\newtheorem{observation}[theorem]{Observation}

\newtheorem{claim}[theorem]{Claim}
\theoremstyle{definition}

\newtheorem{defn}[theorem]{Definition}
\newtheorem{notation}[theorem]{Notation}

\newtheorem*{remark}{Remark}

\newcommand\cref[1]{Corollary~\ref{cor:#1}}

\title{On the number of containments in $P$-free families}
\author{D\'aniel Gerbner$^{a,}$\thanks{Research supported by the J\'anos Bolyai Research Fellowship of the Hungarian Academy of Sciences and the National Research, Development and Innovation Office -- NKFIH under the grant K 116769.}, Abhishek Methuku$^{b,}$\thanks{Research supported by the National Research, Development and Innovation Office - NKFIH under the grant K 116769.}, D\'aniel T. Nagy$^{a,}$\thanks{Research supported by the \'{U}NKP-17-3 New National Excellence Program of the Ministry of Human Capacities and by National Research, Development and Innovation Office - NKFIH under the grant K 116769.}, \\ Bal\'azs Patk\'os$^{a,}$\thanks{Research supported by the National Research, Development and Innovation Office -- NKFIH under the grants SNN 116095 and K 116769.}, M\'at\'e Vizer$^{a,}$\thanks{Research supported by the National Research, Development and Innovation Office -- NKFIH under the grant SNN 116095.} \\
\small $^a$ Alfr\'ed R\'enyi Institute of Mathematics, Hungarian Academy of Sciences\\
\small P.O.B. 127, Budapest H-1364, Hungary.\\
\small $^b$ Central European University, Department of Mathematics\\
\small Budapest, H-1051, N\'ador utca 9.\\
\medskip
\small \texttt{\{gerbner,nagydani,patkos\}@renyi.hu, \{abhishekmethuku,vizermate\}@gmail.com}
\medskip}

\begin{document}
\maketitle
\begin{abstract}
A subfamily $\{F_1,F_2,\dots,F_{|P|}\}\subseteq \cF$ is a copy of the poset $P$ if there exists a bijection $i:P\rightarrow \{F_1,F_2,\dots,F_{|P|}\}$ such that $p\le_P q$ implies $i(p)\subseteq i(q)$. A family $\cF$ is $P$-free, if it does not contain a copy of $P$. In this paper we establish basic results on the maximum number of $k$-chains in a $P$-free family $\cF\subseteq 2^{[n]}$. We prove that if the height of $P$, $h(P) > k$, then this number is of the order $\Theta(\prod_{i=1}^{k+1}\binom{l_{i-1}}{l_i})$, where $l_0=n$ and $l_1\ge l_2\ge \dots \ge l_{k+1}$ are such that $n-l_1,l_1-l_2,\dots, l_k-l_{k+1},l_{k+1}$ differ by at most one. On the other hand if $h(P)\le k$, then we show that this number is of smaller order of magnitude. 


Let $\vee_r$ denote the poset on $r+1$ elements $a, b_1, b_2, \ldots, b_r$, where $a < b_i$ for all $1 \le i \le r$ and let $\wedge_r$ denote its dual. 
For any values of $k$ and $l$, we construct a $\{\wedge_k,\vee_l\}$-free family and we conjecture that it contains asymptotically the maximum number of pairs in containment. We prove that this conjecture holds under the additional assumption that a chain of length 4 is forbidden. Moreover, we prove the conjecture for some small values of $k$ and $l$. We also derive the asymptotics of the maximum number of copies of certain tree posets $T$ of height 2 in $\{\wedge_k,\vee_l\}$-free families $\cF\subseteq 2^{[n]}$. 

\end{abstract}

\section{Introduction}

In extremal set theory, many of the problems considered can be phrased in the following way: what is the size of the largest family of sets that satisfy a certain property. The very first such result is due to Sperner \cite{S1928} which states that if $\cF$ is a family of subsets of $[n]=\{1,2\dots,n\}$ (we write $\cF\subseteq 2^{[n]}$ to denote this fact) such that no pair $F,F'\in \cF$ of sets are in inclusion $F\subsetneq F'$, then $\cF$ can contain at most $\binom{n}{\lfloor n/2\rfloor}$ sets. This is sharp as shown by $\binom{[n]}{\lfloor n/2\rfloor}$ (the family of all $k$-element subsets of a set $X$ is denoted by $\binom{X}{k}$ and is called the $k^{th}$ \textit{layer} of $X$). This was later generalized by Erd\H os \cite{E1945}, who showed that if $\cF \subseteq 2^{[n]}$ does not contain a \textit{chain} of length $k+1$ (i.e. nested sets $F_1\subsetneq F_2 \subsetneq \dots \subsetneq F_{k+1}$), then the size of $\cF$ is at most $\sum_{i=1}^k\binom{n}{\lfloor \frac{n-k}{2}\rfloor +i}$, the sum of the $k$ largest binomial coefficients of order $n$.

If $P$ is a poset, we denote by $\le_P$ the partial order on the elements of $P$. Generalizing Sperner's result, Katona and Tarj\'an \cite{KT} introduced the problem of determining the maximum size of a family $\cF\subseteq 2^{[n]}$ that does not contain sets satisfying some inclusion patterns. 

\begin{defn}

 Let $P$ be a finite poset and $\cF\subseteq 2^{[n]}$. A subfamily $\cG\subseteq \cF$ is a \textit{(weak) copy} of $P$ if there exists a bijection $\phi: P \rightarrow \cG$ such that we have $\phi(x)\subseteq \phi(y)$ whenever $x \le_P y$ holds.

Let $\cF\subseteq 2^{[n]}$ and let $\cP$ be a set of posets. We say that $\cF$ is \textit{$\cP$-free}, if $\cF$ does not contain a copy of $P$ for any $P \in \cP$. Generally, the area of forbidden subposet problems is concerned with determining the quantity

$$La(n,\cP):=\max\{|\cF|:\cF\subseteq 2^{[n]}, ~\text{$\cF$ is $\cP$-free}\}.$$

\end{defn}

If $\cP=\{P\}$ we simply denote the quantity above by $La(n,P).$ We will write $P_k$ for the totally ordered set (path poset) of size $k$ and using this, Erd\H os's above-mentioned result can be formulated as $La(n,P_{k+1})=\sum_{i=1}^k\binom{n}{\lfloor \frac{n-k}{2}\rfloor +i}$.

The value of $La(n,P)$ has been determined precisely or asymptotically for many posets $P$, but still unkown in general. Let us mention some of the results that will be important for us. Let $\wedge_r$ denote the poset on $r+1$ elements $a, b_1, b_2, \ldots, b_r$ where $a > b_i$ for all $1 \le i \le r$. Let $\vee_r$ denote the poset on $r+1$ elements $a, b_1, b_2, \ldots, b_r$ where $a < b_i$ for all $1 \le i \le r$.
 Katona and Tarj\'an \cite{KT} proved that $La(n,\{ \wedge_2, \vee_2 \}) =2\binom{n-1}{\lfloor \frac{n-1}{2}\rfloor}$. They also showed the following.

\begin{theorem}[Katona, Tarj\'an \cite{KT}]
\label{Katona_Tarjan}
For any positive integer $r \ge 2$, we have $$La(n,\vee_r) = La(n,\wedge_r) = \binom{n}{\lfloor n/2\rfloor} \left(1+O\left(\frac{1}{n}\right)\right).$$
\end{theorem}

The Hasse diagram (also known as the cover graph) of a poset $P$ is a graph with vertex set $P$ where $p,q\in P$ are joined by an edge if $p\le_P q$ and there does not exist $r\neq p,q$ with $p\le_P r\le_P q$. A poset is called a \textit{tree poset} if its Hasse diagram is a tree. Bukh \cite{Bukh} generalized Theorem \ref{Katona_Tarjan} to all tree posets by showing that for any tree poset $T$ of height $h(T)$, we have $La(n, T) = (h(T)-1 + o(1)) \binom{n}{\lfloor \frac{n}{2} \rfloor}.$



Recently Gerbner, Keszegh and Patk\'os \cite{GKP} initiated the investigation of counting the maximum number of copies of a poset in a family $\cF\subseteq 2^{[n]}$ that is $\mathcal P$-free. More formally they introduced the following quantity: Let $\cF\subseteq 2^{[n]}$ and $P$ be a poset, then let $c(P,\cF)$ denote the number of copies of $P$ in $\cF$.

\begin{defn} For families of posets $\cP$ and $\cQ$ let

$$La(n,\cP,\cQ):=\max\left\{\sum_{Q\in \cQ}c(Q,\cF): \cF\subseteq 2^{[n]}, ~\text{$\cF$ is $\cP$-free}\right\}.$$

\end{defn}

If either $\cQ=\{Q\}$ or $\cP=\{P\}$, then we simply write $La(n,P,\cQ)$, $La(n,\cP,Q)$, $La(n,P,Q)$. Note that $La(n,\cP)=La(n,\cP,P_1)$.

There are not many results in the literature where
other posets are counted. Katona \cite{K1973} determined the maximum number of $2$-chains (copies
of $P_2$) in a $2$-Sperner ($P_3$-free) family $\cF\subseteq 2^{[n]}$ by showing $La(n,P_3,P_2)=\binom{n}{i_1}\binom{i_1}{i_2}$ where $i_1,i_2$ are chosen such that $n-i_1,i_1-i_2$ and $i_2$ differ by at most one, so $i_1$ is roughly $2n/3$ and $i_2$ is roughly $n/3$. This was reproved in \cite{P2009} and generalized by Gerbner and Patk\'os in \cite{GP2008}, where they proved the following result. To state the theorem and for later purposes we will use the multinomial coefficient: $\binom{n}{l_1,l_2,...,l_k}=\prod_{i=1}^{k+1}\binom{l_{i-1}}{l_i}$, where $i_0=n$ counts the number of $k$-chains $F_1\subsetneq F_2 \subsetneq \dots \subsetneq F_k$ in $2^{[n]}$ with $|F_i|=l_i$.

\begin{theorem}\label{gp}
For any pair $k>l\ge 1 $ of integers we have 
$$La(n,P_k,P_l)=\max_{0\le i_1<i_2<\dots<i_{k-1}\le n}c\left(P_l, \bigcup_{j=1}^{k-1}\binom{[n]}{i_j}\right)=\max_{0\le i_1<i_2<\dots<i_{k-1}\le n}\binom{n}{l_{k-1},\dots,l_1}.$$
Moreover, if $k=l+1$, then the above maximum is attained when the integers $i_1,i_2-i_1,\dots, i_{k-1}-i_{k-2},n-i_{k-1}$ differ by at most one.
\end{theorem}

As the simplest poset apart from $P_1$ is $P_2$, in this paper we focus on the number of pairs in containment in a $P$-free family, i.e. we try to determine or estimate $La(n,P,P_2)$. We prove that the order of magnitude (maybe apart from a polynomial factor) depends on the \textit{height} of $P$ (the length of the longest chain in $P$).

\begin{theorem}\label{szaramajtol}
\textbf{(i)} For any poset $P$ of height at least 3, we have 
$$La(n,P,P_2) = \Theta(La(n,P_3,P_2)).$$

Moreover, 

$$La(n,P_3,P_2)\le La(n,P,P_2)\le La(n,P_{|P|},P_2) \le \left(\left\lfloor\frac{(|P|-1)^2}{4}\right\rfloor+o(1)\right)\cdot La(n,P_3,P_2).$$

\textbf{(ii)} For any connected poset $P$ of height 2 with at least 3 elements, we have
$$\Omega\left(\binom{n}{\lfloor n/2\rfloor}\right)=La(n,P,P_2)=O\left(n2^n\right).$$
\end{theorem}

We believe that the upper bound in part (ii) of Theorem \ref{szaramajtol} can be improved by a factor of $\sqrt{n}$ and we propose the following conjecture.

\begin{conjecture}\label{ketszint} For any poset $P$ of height 2, we have $$La(n,P,P_2)\le O\left(n\binom{n}{\lfloor n/2\rfloor}\right).$$
\end{conjecture}

\begin{remark}
Note that Conjecture \ref{ketszint}, if true, gives the best possible order of magnitude as the family $\binom{[n]}{\lfloor n/2\rfloor}\cup \binom{[n]}{\lfloor n/2\rfloor+1}$ is $K_{2,2}$-free (where $K_{2,2}$ denotes the poset on 4 elements $a,b,c,d$ with $a, b \le c, d$) and it has $\lceil n/2\rceil\binom{n}{\lfloor n/2\rfloor}$ containments. In fact, there are many other families that are $K_{2,2}$-free which have $\Omega(n\binom{n}{\lfloor n/2\rfloor})$ containments, as discussed below.

Let $A(n, 2\delta, k)$ denote the size of the largest family of subsets of $[n]$ such that each subset has size exactly $k$ and the symmetric difference of any pair of distinct sets is at least $2 \delta$. Graham and Sloane \cite{Graham_Sloane} showed that $A(n, 2\delta, k) \ge \frac{1}{q^{\delta-1}}\binom{n}{k}$ where $q$ is any prime power with $q \ge n$. Let $i$ be a fixed integer. Consider the family $\mathcal F$ consisting of all subsets of $[n]$ of size $\lfloor n/2\rfloor$, plus $A(n, 2i, \lfloor n/2\rfloor+i)$ subsets of size $\lfloor n/2\rfloor+i$ where the symmetric difference of any pair of distinct subsets is at least $2i$. The number of containments in $\mathcal F$  is at least $$A(n, 2i, \lfloor n/2\rfloor+i) \binom{\lfloor n/2\rfloor+i}{\lfloor n/2\rfloor} = \Omega \left(\frac{1}{n^{i-1}} \binom{n}{\lfloor n/2\rfloor+i}  n^i \right)= \Omega \left(n\binom{n}{\lfloor n/2\rfloor}\right).$$ Moreover, in $\mathcal F$, any two sets of size $\lfloor n/2\rfloor+i$ intersect in at most $\lfloor n/2\rfloor$ elements, thus $\mathcal F$ is $K_{2,2}$-free.
\end{remark}

Using an inductive argument, we generalize Theorem \ref{szaramajtol} and obtain the following result on the maximum number of $k$-chains in a $P$-free family. 

\begin{theorem}\label{danialt}  Let $l$ be the height of $P$.

\smallskip 

\textbf{(i)} If $l>k$, then 
$$La(n,P,P_k) = \Theta(La(n,P_{k+1},P_k)).$$
Moreover, 

$$La(n,P_{k+1},P_k)\le La(n,P,P_k)\le La(n,P_{|P|},P_k) \le \binom{|P|-1}{k}La(n,P_{k+1},P_k).$$

\textbf{(ii)} If $l\le k$, then $$La(n,P,P_k)=O(n^{2k-1/2}La(n,P_{l},P_{l-1})),$$
and there exists a poset $P$ of height $l$ such that $La(n,P,P_k)=\Theta(La(n,P_l,P_{l-1}))$ holds.
\end{theorem}

Then we focus on specific forbidden posets $P$. By generalizing a construction of Katona and Tarj\'an, we prove the following lower bound on the maximum number of copies of $\wedge_s$ in $\{\wedge_k,\vee_l\}$-free families (note that $P_2$ is the special case $\wedge_1$).

\begin{theorem}
\label{cons}
For any integers $k,l,s$ with $s\le k-1$ we have $$La(n,\{\wedge_k,\vee_l\},\wedge_s) \ge \left(\binom{k-1}{s}\frac{l-1}{k+l-2}+o(1)\right)\binom{n}{\lfloor n/2\rfloor}.$$
\end{theorem}

We conjecture that this lower bound is sharp. 

\begin{conjecture}\label{main1}
For any integers $k,l,s$ with $s\le k-1$ we have $$La(n,\{\wedge_k,\vee_l\},\wedge_s)=\left(\binom{k-1}{s}\frac{l-1}{k+l-2}+o(1)\right)\binom{n}{\lfloor n/2\rfloor}.$$
\end{conjecture}


We can show that the above conjecture holds in the following cases. 

\begin{theorem}\label{easy}
Conjecture \ref{main1} holds in the following cases:

\smallskip  

(i) $s=k-1$,

\smallskip 

(ii) $k,l\le 5$.
\end{theorem}

Moreover, we prove Conjecture \ref{main1} under the additional assumption that $P_4$ is forbidden. We remark that this is indeed a natural assumption, since the extremal families showing the lower bound in Theorem \ref{cons} do not even contain $P_3$ (see Section \ref{construction}).
\begin{theorem}\label{p4free}
For any integers $k,l,s$ with $s\le k-1$ we have, $$La(n,\{\wedge_k,\vee_l,P_4\},\wedge_s)=\left(\binom{k-1}{s}\frac{l-1}{k+l-2}+o(1)\right)\binom{n}{\lfloor n/2\rfloor}.$$
\end{theorem}

\begin{corollary}
Conjecture \ref{main1} holds if either $k$ or $l$ is at most 3. 
\end{corollary}

Now we introduce some notation that is used throughout the paper.
°
\begin{notation}[Comparability graph]
\label{notationComparabilitygraph}
Given a family $\mathcal F \subset 2^{[n]}$, the undirected comparability graph $G_\cF$, and the directed comparability graph $\overrightarrow{G}_\cF$ corresponding to $\mathcal F$, are graphs whose vertex sets are equal to $\cF$ and whose edge sets are defined as follows:
Two vertices $F,F'\in \cF$ are connected by an edge/arc (i.e. directed edge) if and only if $F$ and $F'$ are in containment. In $\overrightarrow{G}_\cF$ the arc is directed from $F$ to $F'$ if $F'\subseteq F$. A component of a family $\cF$ is a subfamily the vertices of $G_\cF$ correspond to form a connected component of $G_\cF$.
\end{notation}

\begin{remark}
Note that the $\wedge_k$-free property is equivalent to the fact that the out-degree $d^+(v)$ of every vertex $v$ in $\overrightarrow{G}_\cF$ is at most $k-1$ and the $\vee_l$-free property is equivalent to the fact that the in-degree $d^-(v)$ of every vertex $v$ in $\overrightarrow{G}_\cF$ is at most $l-1$. The number of pairs of $\cF$ in containment equals the number of edges in $G_\cF$ and the number of arcs in $\overrightarrow{G}_\cF$ (that is the sum of the out-degrees). More generally, we have
\[
c(\wedge_s,\cF)=\sum_{v\in V(\overrightarrow{G}_\cF)}\binom{d^+(v)}{s}.
\]
\end{remark}

\noindent
\textbf{Structure of the paper.} The rest of the paper is organized as follows. In Section \ref{gen_bounds}, we prove Theorem \ref{szaramajtol} and Theorem \ref{danialt}, concerning general bounds on the number of containments and $k$-chains in $P$-free families and we also determine the order of magnitude of $La(n,T,P_2)$ for tree posets $T$ of height 2. In Section \ref{construction}, we prove Theorem \ref{cons} by constructing a $\{\wedge_k,\vee_l\}$-free family with many copies of $\wedge_s$. In Section \ref{overviewpluseasy}, we illustrate our method by proving Theorem \ref{easy}. Finally, in Section \ref{p_4_free_section}, we prove Theorem \ref{p4free} by applying this method.


\section{General bounds}
\label{gen_bounds}

In this section we prove bounds on the maximum number of pairs in containment in $P$-free families. In our proofs we will use the following class of posets: if $a_1,a_2,\dots, a_s$ are positive integers, then the \textit{complete multi-level poset} $K_{a_1,a_2,\dots,a_s}$ has $\sum_{i=1}^sa_i$ elements $p^1_1,p^1_2,\dots p^1_{a_1},\dots,p^s_1,p^s_2,\dots,p^s_{a_s}$ such that $p^j_i< p^{j'}_{i'}$ if and only if $j<j'$. Observe that any poset $P$ of height $l$ is contained in the $l$-level poset $K_{|P|-l+1, |P|-l+1,\dots,|P|-l+1}$.

\begin{proof}[Proof of Theorem \ref{szaramajtol}]
To prove (i), observe first that any $P$-free family is $P_{|P|}$-free, and if the height of $P$ is at least 3, then any $P_3$-free family is $P$-free. This shows the first two inequalities. To prove the last inequality first observe that by Theorem \ref{gp}, it is enough to consider families $\cF$ consisting of $|P|-1$ full levels and determine the value 
\[
\max_{0\le i_1<i_2<\ldots<i_{|P|-1}\le n}\sum_{1\le l<j\le |P|-1}\binom{n}{i_j}\binom{i_j}{i_l}.
\]
As claimed by Theorem \ref{gp}, $\binom{n}{i_j}\binom{i_j}{i_l}$ is maximized when $i_l,i_j-i_l$, and $n-i_j$ differ by at most 1. Furthermore, if 
$i_j\notin ((2/3-\varepsilon)n,(2/3+\varepsilon)n)$ or $i_l\notin ((1/3-\varepsilon)n,(1/3+\varepsilon)n)$, then $\binom{n}{i_j}\binom{i_j}{i_l}=o(\binom{n}{2n/3}\binom{2n/3}{n/3}) = o(La(n,P_3,P_2))$. So, if we consider the graph $G$ with vertex set $\{i_1,i_2,\dots,i_{|P|-1}\}$ where $i_s<i_t$ are joined by an edge if and only if $\binom{n}{i_t}\binom{i_t}{i_s}=\Theta(La(n,P_3,P_2))$, then $G$ is triangle-free. Therefore the number of edges in $G$ is at most $\lfloor \frac{(|P|-1)^2}{4}\rfloor$. This finishes the proof of part (i).

The lower bound of (ii) is given by the family $\cF_{\wedge,\vee}$ constructed by Katona and Tarj\'an \cite{KT}: $$\cF_{\wedge,\vee}=\binom{[n-1]}{\lfloor \frac{n-1}{2}\rfloor }\cup 
\left\{F\cup \{n\}: F\in \binom{[n-1]}{\lfloor \frac{n-1}{2}\rfloor}\right\}.$$ Indeed, all the connected components of the comparability graph of $\cF_{\wedge,\vee}$ have size two, so $\cF_{\wedge,\vee}$ is $\{\wedge_2,\vee_2\}$-free and $c(P_2,\cF_{\wedge,\vee})=\binom{n-1}{\lfloor \frac{n-1}{2}\rfloor } = \Omega\left(\binom{n}{\lfloor n/2\rfloor}\right)$.   

To prove the upper bound of (ii) observe that if a family $\cF\subseteq 2^{[n]}$ is $P$-free, then in particular it is $K_{|P|-1,|P|-1}$-free, so we obtain 
\[
La(n,P,P_2)\le La(n,K_{|P|-1,|P|-1}, P_2).
\] 
Therefore to finish the proof it is enough to show $La(n,K_{s,s}, P_2)\le O_s(n2^n)$ for any given integer $s$. Let $\cG \subseteq 2^{[n]}$ be a $K_{s,s}$-free family and for any pair $G,G'\in \cG$ with $G\subset G'$ let us define $M=M(G,G')$ to be a set with $G\subseteq M\subseteq G'$ which is maximal with respect to the property that there exist at least $s$ sets $G_1,G_2,\dots, G_s\in \cG$ with $M\subsetneq G_i$ $i=1,2,\dots,s$.

First let us consider those pairs $G\subset G'$ for which $M$ cannot be defined. This means that $G$ is contained in at most $s-1$ other sets of $\cG$ and thus the number of such pairs is at most $(s-1)|\cG|=O_s(2^n)$.

Next, for a fixed $M\in 2^{[n]}$, let us consider the pairs $G\subset G'$ with $M=M(G,G') = G'$. Note that since $M$ is contained in $s$ sets of $\cG$ (and $\cG$ is $K_{s,s}$-free), $M$ can contain at most $s-1$ sets from $\cG$. In particular, the number of pairs $G\subset G'$ with $M(G,G')=G'$ is at most $(s-1)|\cG|=O_s(2^n)$.

Finally, let us consider the pairs $G\subset G'$ with $M=M(G,G')\subsetneq G'$. Each such $G'$ contains a set $M'=M\cup \{x\}$ with $x\notin M$. The number of such $M'$ is $|G'|-|M|\le n$, and for a given $M'$, the number of $G'$ containing $M'$ is at most $s$ (namely, $M'$ and $s-1$ other sets from $\cG$) as otherwise $M'$ would be fit to play the role of $M(G,G')$. So the number of $G'$ containing $M$ is at most $sn$. (Moreover, $M$ can contain at most $s-1$ sets from $\cG$, so there are at most $(s-1)$ choices for $G$.) Therefore the number of pairs $G\subset G'$ with $M(G,G')=M$ is at most $(s-1) \cdot sn$. Summing over all sets $M$ and adding the other types of pairs in containment we obtain
\[
c(P_2,\cG)\le O_s(2^n)+s(s-1)n2^n = O_s(n2^n).
\]
which finishes the proof of the upper bound of (ii).
\end{proof}

\bigskip





\bigskip
Before turning to the proof of Theorem \ref{danialt}, let us determine the order of magnitude of $La(n,T,P_2)$  for any tree poset $T$ of height 2. 

\begin{proposition}
For any tree poset $T$ of height 2 with $|T|\ge 3$, we have $$La(n,T,P_2) = \Theta\left(\binom{n}{\lfloor n/2\rfloor}\right).$$
\end{proposition}

\begin{proof}
The lower bound follows from the lower bound of Theorem \ref{szaramajtol} (ii).

Now we prove the upper bound. Note that a $T$-free family $\mathcal F$ does not contain a chain of length $\abs{T}$. Therefore, we can partition $\mathcal F$ into antichains $\mathcal A_1, \mathcal A_2, \ldots, \mathcal A_m$, such that $m \le \abs{T}-1$ where $\mathcal A_i$ is the family of minimal elements in $\mathcal F \setminus (\cup_{k =1}^{i-1} \mathcal A_k)$ for every $i$. Sperner's theorem implies $|\mathcal A_i|\le \binom{n}{\floor{n/2}}$ for every $i$.

For $i < j$, let $n_{i,j}$ be the number of containments $A \subset B$ such that $A \in \mathcal A_i, B \in \mathcal A_j$. Notice that it is impossible that $A \subset B$ for $A \in \mathcal A_i$ and $B \in \mathcal A_j$ when $i > j$, so the number of $P_2$'s in $\mathcal F$ is $\sum_{i < j} n_{i,j}$.

We claim that for any $1 \le i < j \le m$, we have $n_{i,j} \le \abs{T}(\abs{\mathcal A_i}+\abs{\mathcal A_j})$. Indeed, suppose otherwise, and consider the comparability graph $G=G_{\cA_i\cup \cA_j}$. Then in $G$, there are at least $\abs{T} \abs{V(G)}$ edges, so the average degree in $G$ is at least $2\abs{T}$. It is easy to find a subgraph $G'$ of $G$ with minimum degree at least $\abs{T}$, and one can then embed $T$ greedily into $G'$, giving an embedding of $T$ into $\mathcal F$, a contradiction.  

Therefore, the number of $P_2$'s in $\mathcal F$ is $$\sum_{1\le i < j\le m} n_{i,j} \le \sum_{1\le i < j\le m} \abs{T}(\abs{\mathcal A_i}+\abs{\mathcal A_j})  < \abs{T}^3 \binom{n}{\floor{n/2}}.$$
\end{proof}

\begin{proof}[Proof of Theorem \ref{danialt}] The proof of \textbf{(i)} is similar to the proof of \textbf{(i)} in Theorem \ref{szaramajtol}. Observe first that any $P$-free family is $P_{|P|}$-free and any $P_l$-free family is $P$-free. This shows the first two inequalities. To prove the last inequality, consider the canonical partition of a $P$-free family $\cF$ into at most $|P|$ antichains. We can choose $k$ of them $ \binom{|P|}{k}$ ways, and in each of the resulting $k$-Sperner families there are at most $La(n,P_{k+1},P_k)$ $k$-chains. Note that we counted every $k$-chain in $\cF$ once.

\smallskip

To prove the bound in \textbf{(ii)}, let $K$ and $K'$ be the complete $l$-level and $(l-1)$-level posets with parts of size $s=|P|-1$.  Observe that if a family $\cF\subseteq 2^{[n]}$ is $P$-free, then in particular it is $K$-free,  so we obtain
$La(n,P,P_k)\le La(n,K, P_k)$.
We use induction on $k$. The base case $k=2$ is given by  Theorem \ref{szaramajtol}, and we note the proof is similar to the proof of Theorem \ref{szaramajtol}. Let us also mention that the statement is trivial for $l=1$, hence we can assume $l\ge 2$.

Let $\cG \subseteq 2^{[n]}$ be a $K$-free family and consider a $k$-chain $\cC$ consisting of the sets $G_1\subset \dots \subset G_k$ in $\cG$. Let $M=M(\cC)$ be a set with $G_1\subset \dots \subset G_{k-1}\subseteq M\subseteq G_k$ which is maximal with respect to the property that there exist at least $s$ sets $H_1,H_2,\dots, H_s\in \cG$ with $M\subsetneq H_i$ $i=1,2,\dots,s$. Let $\cM=\{M(\cC)~|~\cC~ \text{is a k-chain in}\ \cG\}$.

We will upper bound the number of $k$ chains $\cC=\{G_1\subset G_2 \subset \dots \subset G_k\}$ (with $G_i \in \mathcal G$) in each of the following 3 cases separately: $M(\cC)$ is not defined at all, $M(\cC)=G_k$ and finally $M(\cC)\subsetneq G_k$.

First let us consider those $k$-chains for which $M$ cannot be defined. This means that $G_{k-1}$ is contained in at most $s-1$ other sets of $\cG$ and thus the number of such $k$-chains is at most $(s-1)$ times the number of $(k-1)$-chains. If $l\le k-1$, then by induction, the number of $(k-1)$-chains is at most $O(n^{2k-2-1/2}La(n,P_{l},P_{l-1}))$, otherwise $l = k$ (recall that we assumed $l \le k$) and then \textbf{(i)} shows that the number of $(k-1)$-chains is $O(La(n,P_{l},P_{l-1}))$.

Next, consider a fixed set $M\in \cM$. Note that $\cG$ is $K$-free and $M$ is contained in $s$ sets of $\cG$, therefore $M$ cannot contain $K'$ in $\cG$. In particular, the number of chains of length $k-1$ contained in $M$ is at most $O(n^{2k-2-1/2}La(|M|,P_{l-1},P_{l-2}))$ by induction. In particular, the number of $k$-chains $\cC=\{G_1\subset G_2 \subset \dots \subset G_k\}$ for which $M(\cC)=G_k$ is $s \cdot O(n^{2k-2-1/2}La(|M|,P_{l-1},P_{l-2}))$.

Finally, let us now count the chains $\cC$ with $M=M(\cC)\subsetneq G_k$. Given such a chain $G_1\subset G_2 \subset \dots \subset G_k$, we know that $G_k$ contains a set $M'=M\cup \{x\}$ with $x\notin M$, as $M$ is its proper subset. The number of such sets $M'$ is $n-|M|\le n$ and for a given $M'$ the number of sets in $\cG$ containing $M'$ is at most $s$ (namely, $M'$ and $s-1$ other sets from $\cG$), as otherwise $M'$ would be fit to play the role of $M(\cC)$. It means, given the bottom $k-1$ sets in a chain that are contained in $M$, there are at most $sn$ ways to pick $G_k$. Thus the number of $k$-chains $\cC$ with $M=M(\cC)$ is at most $sn$ times the number of chains of length $k-1$ contained in $M$, which is at most $$O(n^{2k-2-1/2})La(|M|,P_{l-1},P_{l-2}),$$  by induction.

The total number of $k$-chains for which $M(\cC)$ could be defined is then 
\begin{equation}
\label{mcdefined}
\sum_{M\in \cM}O(sn \cdot n^{2k-2-1/2}La(|M|,P_{l-1},P_{l-2}))=O(n^{2k-1-1/2})\sum_{i=0}^n \sum_{M\in \binom{n}{i}\cap \cM}La(i,P_{l-1},P_{l-2}).
\end{equation}

\begin{claim} For any $i\le n$ we have $$\sum_{M\in \binom{n}{i}\cap \cM}La(i,P_{l-1},P_{l-2})\le La(n,P_l,P_{l-1}).$$ 
\end{claim}

\begin{proof} By Theorem \ref{gp}, there are integers $i_1,\dots,i_{l-2}$ such that $La(i,P_{l-1},P_{l-2})$ is the number of $(l-2)$-chains in the family consisting of all sets of sizes $i_1,\dots,i_{l-2}$ in $2^{[i]}$. Therefore, $\sum_{M\in \binom{n}{i}}La(i,P_{l-1},P_{l-2})$ is equal to the number of $(l-1)$-chains consisting of sets of size $i_1,\dots,i_{l-2},i$. As these $l-1$ levels do not contain $P_l$, the number of $(l-1)$-chains is at most $La(n,P_l,P_{l-1})$ by definition.
\end{proof}

Using the above claim and \eqref{mcdefined}, we obtain that the number of $k$-chains for which $M(\cC)$ could be defined is $\sum_{i=0}^nO(n^{2k-1-1/2})La(n,P_l,P_{l-1})=O(n^{2k-1/2})La(n,P_l,P_{l-1})$. We already obtained that the number of $k$-chains for which $M(\cC)$ could not be defined is $O(n^{2k-2-1/2})La(n,P_{l},P_{l-1})$, which finishes the proof of the upper bound in \textbf{(ii)}.

Now we prove the remaining part of \textbf{(ii)}. Let $K=K_{s,s,\dots,s}$ be the complete $l$-level poset with $s> k-l+1$. Let $i_1,i_2,\dots, i_{l-1}$ be integers such that $n-(k-l+1)-i_1,i_1-i_2,\dots,i_{l-2}-i_{l-1},i_{l-1}$ differ by at most 1. By Theorem \ref{gp}, we know that the family $\cG'=\cup_{j=1}^{l-1}\binom{[n-(k-l+1)]}{i_j}$ realizes $La(n-(k-l+1),P_l,P_{l-1})$. Since $\cG'$ is $(l-1)$-Sperner, if we add sets to $\cG'$ such that no two $G,G'\in \binom{[n-(k-l+1)]}{i_{l-1}}$ are contained in the same newly added sets, then the resulting family will be $K$-free. If we add the sets 
\[
\left\{[n-(k-l+1)+1,n-(k-l+1)+j]\cup G: j\in [k-l+1], G\in \binom{[n-(k-l+1)]}{i_{l-1}}\right\}
\]
and denote the resulting family by $\cG$, then we have 
\[
c(\cG,P_k)=c(\cG',P_{l-1})=La(n-(k-l+1),P_l,P_{l-1})=\Omega_{k,l}(La(n,P_l,P_{l-1})).
\]
This finishes the proof.
\end{proof}

\section{Proof of Theorem \ref{cons}: Construction}
\label{construction}

We need to show a $\{ \wedge_k,\vee_l\}$-free family $\cF_{n,k,l}\subseteq 2^{[n]}$ with $c(\cF_{n,k,l},\wedge_s)=\left(\binom{k-1}{s}\frac{l-1}{k+l-2}+o(1)\right)\binom{n}{\lfloor n/2\rfloor}$. Let us partition $[n]$ into $A_1,A_2,\dots, A_m,A_{m+1}$ with $m=\lfloor \frac{n}{k+l-3}\rfloor$ and $|A_i|=k+l-3$ for all $i=1,2,\dots,m$. Let us define the family $$\cF_{n,k,l}=\cup_{j=1}^m\cF_{n,k,l,j}^+\cup \cup_{j=1}^m\cF_{n,k,l,j}^-, \ \textrm{ where}$$
\begin{equation*}
\begin{split}
\cF_{n,k,l,j}^+= & \left\{F\in \binom{[n]}{\lfloor n/2\rfloor +1}:  |F\cap A_j|=k-1, \text{ and } \hskip 0.1truecm \forall i<j \text{ we have } \hskip 0.1truecm |F\cap A_i|\neq k-2,\ k-1\right\}, \\
\cF_{n,k,l,j}^-=& \left\{F\in \binom{[n]}{\lfloor n/2\rfloor }: |F\cap A_j|=k-2, \hskip 0.1truecm \text{ and } \forall i<j \text{ we have } \hskip 0.1truecm |F\cap A_i|\neq k-2,\ k-1\right\}.
\end{split}
\end{equation*}

\begin{remark}
Observe that if $k=l=2$, then $k-2=0,k-1=1, k+l-3=1$. In particular, $|A_1|=1$ and let $A_1 = \{ 1\}$. So, $\cF^+_{n,2,2,1}=\{F\in \binom{[n]}{\lfloor n/2\rfloor +1}:  1\in F\}$ and $\cF^-_{n,2,2,1}=\{F\in \binom{[n]}{\lfloor n/2\rfloor }:  1\notin F\}$. As every set $F$ intersects $A_1 = \{ 1\}$ in either 0 or 1 element, all other $\cF^+_{n,2,2,j}$'s and $\cF^-_{n,2,2,j}$'s are empty and thus $\cF_{n,2,2}$ is equal (up to permutations of the ground set) to the family $\cF_{\wedge,\vee}$ of Katona and Tarj\'an described in the proof of Theorem \ref{szaramajtol}.
\end{remark} 

\bigskip

First we show that $\cF_{n,k,l}$ is $\{\wedge_k,\vee_l\}$-free. As sets in $\cF_{n,k,l}$ are of size $\lfloor n/2 \rfloor$ and $\lfloor n/2 \rfloor +1$, if $F\subset G$ holds for some $F,G\in \cF_{n,k,l}$, then we have $F\in \cup_{j=1}^m\cF_{n,k,l,j}^-,G\in \cup_{j=1}^m\cF_{n,k,l,j}^+$. Let $j(F)$ and $j(G)$ be the indices with $F\in \cF_{n,k,l,j(F)}^-, \ G\in \cF_{n,k,l,j(G)}^+$. We claim that $j(F)=j(G)$ holds. Indeed,  if $j(G)< j(F)$ (the case $j(G)>j(F)$ is similar), then $F\subset G$ implies $|F\cap A_{j(G)})|\le k-3$ (indeed, if the intersection size is equal to  $k-2$, then $j(F)=j(G)$ would hold, if it is $k-1$, it would mean that $F$ is not in the family, and if it is more than $k-1$ it would mean $G$ does not contain $F$) and thus $$|F\cap ([n]\setminus A_{j(G)})|\ge \lfloor n/2\rfloor-(k-3)>\lfloor n/2\rfloor +1-(k-1)=|G\cap ([n]\setminus A_{j(G)})|$$
contradicting $F\subset G$.

Furthermore, if $j(F)=j(G)$, then $|F\cap ([n]\setminus A_{j(G)})|=|G\cap ([n]\setminus A_{j(G)})|$, so $G$ contains those sets of $\cF_{n,k,l}$ that are of the form $G\setminus \{a\}$ with $a\in G\cap A_{j(G)}$. By definition of $\cF_{n,k,l,j}^+$ there are $k-1$ such $a$'s. Similarly, any $F$ is contained in those sets of $\cF_{n,k,l}$ that are of the form $F\cup \{a\}$ with $a\in A_{j(F)}\setminus F$. By definition of $\cF_{n,k,l,j}^-$ there are $k+l-3-(k-2)=l-1$ such $a$'s. Therefore, $\cF_{n,k,l}$ is indeed $\{\wedge_k,\vee_l\}$-free. Moreover, we have
\[
c(\wedge_s,\cF_{n,k,l})=\binom{k-1}{s}|\cup_{j=1}^m\cF_{n,k,l,j}^+|.
\]
Therefore it remains to show that $\cF^+_{n,k,l}:=\cup_{j=1}^m\cF_{n,k,l,j}^+$ has size $(\frac{l-1}{k+l-2}+o(1))\binom{n}{\lfloor n/2\rfloor + 1}$. To do this, let us introduce 
\[
p_1:=\frac{\binom{k+l-3}{k-1}}{2^{k+l-3}} \hskip 1truecm \text{and} \hskip 1truecm p_2:=\frac{2^{k+l-3}-\binom{k+l-3}{k-1}-\binom{k+l-3}{k-2}}{2^{k+l-3}}.
\]
Let us bound $|\cF^+_{n,k,l,j}|$ from below: given $H\subseteq \cup_{i=1}^jA_i$ with $|H\cap A_j|=k-1$ and $|H\cap A_i|\neq k-1,k-2$ let $$\cF_H := \{F\in \binom{[n]}{\lfloor n/2\rfloor +1}: F\cap \cup_{i=1}^jA_i=H\}$$ Clearly, $\cF^+_{n,k,l,j}$ is the union of $\cF_H$'s over all $H$ satisfying the required intersection property. Also, we have $$|\cF_H|=\binom{n-j(k+l-3)}{\lfloor n/2\rfloor +1-|H|}.$$
Observe that $0\le |H|\le j(k+l-3)$. So $j\le \log n$ implies that for any $\varepsilon>0$ and $H$ with the above property if $n$ is large enough, then we have $$2^{j(k+l-3)}\binom{n-j(k+l-3)}{\lfloor n/2\rfloor +1-|H|}\ge (1-\varepsilon)\binom{n}{\lfloor n/2\rfloor+1}.$$ 
As the number of sets $H$ possessing the required intersection property is $(p_22^{k+l-3})^{j-1}p_12^{k+l-3}$, we obtain
\[
|\cF^+_{n,k,l}|\ge (1-\varepsilon)\binom{n}{\lfloor n/2\rfloor+1}\sum_{j=1}^{\log n}p_2^{j-1}p_1\ge (1-\varepsilon')\frac{p_1}{1-p_2}\binom{n}{\lfloor n/2\rfloor+1}.
\]
This finishes the proof as $$\frac{p_1}{1-p_2}=\frac{l-1}{k+l-2}.$$

\qed 

\section{Overview of our method and Proof of Theorem \ref{easy}}
\label{overviewpluseasy}

\subsection{Proof of Theorem \ref{easy} (i) and a lemma}
\begin{proof}[Proof of Theorem \ref{easy} (i)]
Let $\cF\subseteq 2^{[n]}$ be a $\{\wedge_k,\vee_l\}$-free family and let us consider $\overrightarrow{G}_{\cF}$, its directed comparability graph. Let $A$ denote the set of vertices with out-degree $k-1$. As $\cF$ is $\wedge_k$-free, this is the maximum out-degree and $c(\wedge_{k-1},\cF)=|A|$. Observe that $d^-(a)=0$ for any $a\in A$, as if $(b,a)\in E(\overrightarrow{G}_{\cF})$ (i.e., there is a directed edge from $b$ to $a$), then $d^+(b)\ge k$ would hold (every out-neighbor of $a$ would be an out-neighbor of $b$ in addition to $a$) contradicting the $\wedge_k$-free property of $\cF$. This implies that $A$ is an independent set in $\overrightarrow{G}_{\cF}$ and there is no edge from $B=V(\overrightarrow{G}_{\cF})\setminus A$ to $A$. As $\cF$ is $\vee_l$-free also, the in-degree of every vertex in $B$ is at most $l-1$. Therefore double counting the edges between $A$ and $B$ we obtain
\[
(k-1)|A|\le (l-1)|B|=(l-1)(|\cF|-|A|).
\]
Rearranging gives
\[
c(\wedge_{k-1},\cF)=|A|\le \frac{l-1}{k+l-2}|\cF|\le \left(\frac{l-1}{k+l-2}+o(1)\right)\binom{n}{n/2}
\]
where for the last inequality we used Theorem \ref{Katona_Tarjan}. One can see that this bound is sharp by putting $s = k-1$ in Theorem \ref{cons}
\end{proof}


Below we show that if the construction $\cF_{n,k,l}$ (defined in Section \ref{construction}) maximizes the number of containments, then it maximizes the number of certain subposets in the family as well.

\begin{lemma}
  Assume that for some fixed integers $k,l$ we have $$La(n,\{\wedge_k,\vee_l\},  P_2)=\left(\frac{(k-1)(l-1)}{k+l-3}+o(1)\right)\binom{n}{\lfloor n/2\rfloor}.$$ Let $P$ be a poset on two levels not containing $\wedge_k$ and $\vee_l$ as a subposet. Also assume that the Hasse diagram of $P$ is a tree containing no 6-edge path. Then the number of copies of $P$ in a $\{\wedge_k,\vee_l\}$-free family $\cF\subset 2^{[n]}$ is asymptotically maximized by the construction $\cF_{n,k,l}$.
\end{lemma}

\begin{proof}
We already showed that every set $v\in\cF^+_{n,k,l}$ has $d^+(v)=k-1$ and every set $v\in\cF^-_{n,k,l}$ has $d^-(v)=l-1$. (Actually, $\cF_{n,k,l}$ consists of several copies of the poset formed by the $(k-2)$-th and $(k-1)$-th levels of the Boolean lattice of order $k+l-3$.)

We try to find an embedding $f$ of $P$ into a $\{\wedge_k,\vee_l\}$-free family $\cF\subset 2^{[n]}$. First, pick an edge of the Hasse diagram of $P$ and embed its endpoints to two sets of $\cF$. We have as many possibilities as the number of containments among sets in $\cF$. By Theorem \ref{cons} and our assumption, this is asymptotically maximized by $\cF=\cF_{n,k,l}$.

We continue the embedding of the points of $P$ as follows. We pick a point $p\in P$ that is already embedded, and embed all points that are connected to $p$ in the Hasse diagram and are not already embedded. If $p$ is in the upper (or lower) level of $P$, then the number of our possibilities is maximal if $d^+(f(p))=k-1$ (or $d^-(f(p))=l-1$, respectively). This is true if $\cF=\cF_{n,k,l}$. Moreover, if $\cF=\cF_{n,k,l}$ we will not get stuck during the embedding when trying to embed a point $p\in P$ into a set $F\in\cF$ that is already used, because the comparability graph of $\cF_{n,k,l}$ has no cycle shorter than 6 edges and the Hasse diagram of $P$ contains no 6-edge path.
\end{proof}

We can use the above lemma in the case $P=\wedge_s$ ($s\le k-1$) to obtain the following:

\begin{corollary}
\label{reductingtop2}
If for some fixed integers $k,l$ we have $$La(n,\{\wedge_k,\vee_l\}, P_2)=\left(\frac{(k-1)(l-1)}{k+l-3}+o(1)\right)\binom{n}{\lfloor n/2\rfloor},$$ then for any $s\le k-1$ we have
$$La(n,\{\wedge_k,\vee_l\},\wedge_s)=\left(\binom{k-1}{s}\frac{l-1}{k+l-3}+o(1)\right)\binom{n}{\lfloor n/2\rfloor}.$$
\end{corollary}

By Corollary \ref{reductingtop2}, notice that in order to prove part (ii) of Theorem \ref{easy}, it suffices to prove that $La(n,\{\wedge_k,\vee_l\}, P_2)=\left(\frac{(k-1)(l-1)}{k+l-3}+o(1)\right)\binom{n}{\lfloor n/2\rfloor}$ for $k, l \le 5$. We will prove this in the next subsection, by introducing a new method.


\subsection{Overview of our method}
In this subsection we illustrate our method by showing part (ii) of Theorem \ref{easy}. We start by giving a detailed proof of
\begin{equation}
\label{eq:easy1}
La(n,\{\wedge_4,\vee_4\}, P_2)=\frac{3}{2}\binom{n}{\lfloor n/2\rfloor} \left(1+O\left(\frac{1}{n}\right)\right).
\end{equation}
Then applying Corollary \ref{reductingtop2}, this would complete the proof of Theorem \ref{easy} (ii) in the case $k = l = 4$ (as noted before). Note that the lower bound in \eqref{eq:easy1} follows from Theorem \ref{cons}. Now we show the upper bound. 

\begin{proof}[Proof of Theorem \ref{easy} (ii) for $k = l = 4$]

Let $\mathcal F$ be a $\{\wedge_4,\vee_4\}$-free family of subsets of $[n]$. 
Let $G_{\mathcal F}$ be the comparability graph  corresponding to $\mathcal F$ (see Notation \ref{notationComparabilitygraph}).
Let $\mathcal K$ be the set consisting of the vertex sets of all (maximal) connected components of $G_{\mathcal F}$. Clearly, 
\begin{equation}
\label{sizeoforiginalF}
\abs{\mathcal F} = \sum_{C \in \mathcal K} \abs{C}.
\end{equation}
For the sake of brevity, we refer to a connected component just by its vertex set. Observe that, by definition, for any two connected components $C_1, C_2 \in \mathcal K$, the elements $v_1 \in C_1$ and $v_2 \in C_2$ are unrelated. 
For each connected component of $G_{\mathcal F}$ with vertex set $C \in \mathcal K$, let $G_C$ be the subgraph of $G_{\mathcal F}$ induced by $C$.
Then the number of containments in $\mathcal F$ is 

\begin{equation}
\label{total_edges}
c(P_2, \mathcal F) = \sum_{C \in \mathcal K} \abs{E(G_C)}
\end{equation}

Suppose that for each connected component $C \in \mathcal K$, the following inequality holds.
\begin{equation}
\label{edgesversusvertices}
\abs{E(G_C)} \le \frac{3}{2} \abs{C}.
\end{equation}
Then by \eqref{sizeoforiginalF} and \eqref{total_edges} we would get that $c(P_2, \mathcal F) \le \frac{3}{2} \abs{\mathcal F}$, and by Theorem \ref{Katona_Tarjan} we are done. But this is not necessarily true: Consider the connected component $S$ with elements $a,b,c,d,e$ where $a, b < v < d, e$. Then $G_S$ has 8 edges but $\abs{S} = 5$, so \eqref{edgesversusvertices} does not hold. In fact, we claim that if there is a connected component in $G_{\mathcal F}$ for which \eqref{edgesversusvertices} does not hold, then it must be isomorphic to $G_S$. Indeed, notice that for a connected component $C$, each vertex has degree at most 3 in $G_C$, then the inequality \eqref{edgesversusvertices} trivially holds. So we can assume that $G_C$ must have a vertex $v$ of degree at least 4. Now, it is impossible that $v$ is more than or less than 4 elements in $C$ since $\mathcal F$ is $\{\wedge_4,\vee_4\}$-free. Moreover, it is also impossible that $v$ is more than 3 elements $a,b,c$ and less than an element $d$, since then $d$ is above 4 elements of $C$, a contradiction. Similarly, it is impossible that $v$ is less than 3 elements and more than an element of $C$. Therefore, the only possibility is that $v$ is more than exactly 2 elements $a, b$ and less than exactly 2 elements $d, e$. Moreover, it is easy to check that in this case, the elements $a,b,c,d$ cannot be related to any other elements, proving our claim.

In order to fix the above mentioned problem, our idea is the following. We add some sets to $\mathcal F$ to produce a new family $\mathcal G$: Consider each subfamily $\mathcal S_i = \{A_i, B_i, V_i, D_i, E_i\}$ (with $A_i, B_i \subset V_i\subset D_i, E_i$ and $1 \le i \le m$) of $\mathcal F$ corresponding to a connected component of $G_{\mathcal F}$ which is isomorphic to $G_S$, and add exactly one set $F_i \in [A_i, D_i] \setminus \{V_i\}$ to it, where $[A_i, D_i]:=\{V \in 2^{[n]}: A_i \subsetneq V \subsetneq D_i\}$. Let $\mathcal S'_i = \mathcal S_i \cup \{F_i\}$ and let $\cG = \cF \cup \{F_1, F_2, \ldots, F_m\}$. We claim that each newly added set $F_i \in \mathcal S'_i$ is unrelated to all the other sets $L \in \mathcal G \setminus \mathcal S'_i$. Indeed, if $L  \in \mathcal G \setminus \mathcal F$ -- i.e., say $L = F_j$ for some $j \not = i$ -- then either $A_i$ or $B_i$ is related to $A_j$ or $B_j$, contradicting the assumption that they correspond to elements in different components of $G_{\mathcal F}$. If $L \in \mathcal F$, then again either $A_i$ or $B_i$ is related to $L$, a contradiction. This claim shows that the connected components of the comparability graph $G_{\mathcal G}$ corresponding to $\mathcal G$ are the same as those of $G_{\mathcal F}$ except that the components isomorphic to $G_S$ in $G_{\mathcal F}$ are replaced by $G_{S'}$.  Therefore, any set of $\mathcal G$ is related to at most one of the newly added sets $F_i$. So $\mathcal G$ is $\{\wedge_5,\vee_5\}$-free.  

Let $\mathcal K'$ be the set consisting of the vertex sets of connected components of $G_{\mathcal G}$. For each connected component $C \in  \mathcal K' \cap \mathcal K$ the inequality \eqref{edgesversusvertices} holds.  As we already noted, for each connected component $C \in \mathcal K' \setminus \mathcal K$, $G_C$ is isomorphic to $G_{S'}$.
Moreover, since $\abs{S'} = 6$ and $\abs{E(G_S)} = 8$, we have $\abs{E(G_S)} \le \frac{3}{2} \abs{S'}$. Therefore, by \eqref{total_edges}, $$ c(P_2, \mathcal F) = \sum_{C \in \mathcal K} \abs{E(G_C)} \le \frac{3}{2} \sum_{C \in \mathcal K'} \abs{C} = \frac{3}{2} \abs{\mathcal {G}}.$$
On the other hand, $\abs{\mathcal{G}} = \binom{n}{\lfloor n/2\rfloor} (1+O(\frac{1}{n}))$ by Theorem \ref{Katona_Tarjan} since $\mathcal{G}$ is $\{\wedge_5,\vee_5\}$-free. Combining this with the above inequality, the proof is complete.
\end{proof}

\bigskip

We leave the proof of the cases $k=4,l=5$ and $k=5,l=4$ to the reader and sketch the proof of the most technical of the cases: $k=5,l=5$. 

\begin{proof}[Proof of Theorem \ref{easy} (ii) for $k = l = 5$]
Again the lower bound follows from Theorem \ref{cons}. Now we show the upper bound. Consider a $\{\wedge_5,\vee_5\}$-free family $\cF$ and its comparability graphs $G_\cF$ and $\overrightarrow{G}_\cF$. Recall that, for any $F \in \cF$, $d(F)$ denotes the degree of $F$ in ${G}_\cF$ and $d^+(F), d^-(F)$ denote the out degree and in degree of $F$ in $\overrightarrow{G}_\cF$ respectively. We want to show the number of edges in $G_\cF$ is at most $(2+O(1/n))\binom{n}{n/2}$. (Then we would be done by applying Corollary ~\ref{reductingtop2}.)  Equivalently, we would like to show that $\sum_{F\in \cF}d(F)=\sum_{F\in \cF}(d^+(F)+d^-(F))$ is at most $(4+O(1/n))\binom{n}{n/2}$.  We say that a set $F\in \cF$ is \textit{problematic} if $d(F)\ge 5$. As $\cF$ is $\{\wedge_5,\vee_5\}$-free, trivially, we have $d(F)\le 8$ for any $F\in \cF$. Moreover, notice that $F$ cannot be a problematic set if $d^+(F)=0$ or $d^-(F)=0$. So there is a set $D$ contained in $F$, which implies that $F$ is contained in at most $3$ sets (otherwise, $D$ would be contained in $5$ sets, contradicting the $\vee_5$-free property of $\cF$). By a symmetric argument, we can conclude that $F$ contains at most $3$ sets, implying  $d(F)\le 6$. (Similar reasoning shows that if $\cF$ is $\{\wedge_k,\vee_l\}$-free, then $d(F)\le k-2+l-2$ holds.) Let $\cF_1$ denote the family of problematic sets. As noted before, for any $F \in \cF_1$ we have $1 \le d^+(F), d^-(F) \le 3$ and consider $U_1,...,U_{d^-(F)}$ with $F \subseteq U_i$ for $1\le i \le d^-(F)$, and $D_1,...,D_{d^{+}(F)}$ with $D_j \subseteq F$ for $1 \le j \le d^{+}(F)$. Let us define the following family $$\cN(F):=\{(U_i\setminus F)\cup D_j : 1\le i \le d^+(F), 1 \le j \le d^{-}(F)\}.$$ As the $U_i$'s are distinct supersets of $F$ and the $D_j$'s are dstinct subsets of $F$ we obtain that $|\cN(F)|=d^-(F)d^{+}(F)$ and $\cN(F) \cap \{U_1,...,U_{d^+(F)},D_1,...,D_{d^{-}(F)}\}=\emptyset$.

\begin{claim}\label{pupu1}
For any $F\in \cF_1$ we have $\cN(F)\setminus \cF\neq \emptyset$.
\end{claim}

\begin{proof}
For any $F\in \cF_1$ we have $\max\{d^+(F),d^-(F)\}\ge 3$ (in fact, $\max\{d^+(F),d^-(F)\} = 3$ since $d^+(F),d^-(F) \le 3$). Suppose $d^-(F)\ge 3$ (the case when $d^+(F)\ge 3$ is similar) and with the above notation let us consider the sets $(U_i\setminus F)\cup D_1$ ($i=1,2,3$). They are all distinct and do not contain $F$. If they all belong to $\cF$, then $D_1$ is contained in at least 7 sets contradicting the $\vee_5$-free property of $\cF$.
\end{proof}

\begin{claim}\label{pupu2}
For any pair $F,F'\in \cF_1$ we have $(\cN(F)\setminus \cF)\cap(\cN(F')\setminus \cF)=\emptyset$.
\end{claim}

\begin{proof}
Suppose to the contrary that $G\in (\cN(F)\setminus \cF)\cap (\cN(F')\setminus \cF)$ for some $F,F'\in \cF_1$.

\medskip

\textsc{Case I:} $F,F'$ are in containment, say $F\subsetneq F'$.

\smallskip

In this case the component of $F,F'$ in $\cF$ consists of six sets $D_1,D_2\subset F\subset F'\subset U_1,U_2$. Observe that every set in $\cN(F)$ contains an element in $F' \setminus F$ while this does not hold for any set in $N(F')$.

\medskip

\textsc{Case II:} $F,F'$ are not in containment.

\smallskip

Recall that we have $\max\{d^+(F),d^-(F)\} = 3$. In this case, we may assume $d^-(F)=3$ as the case $d^+(F)=3$ is symmetric. Let $U_1,U_2,U_3\in \cF$ and $U_1',U_2'\in \cF$ contain $F$ and $F'$ respectively and let $D_1,D_2,D_1',D_2'\in \cF$ with $$D_1,D_2\subset F,\ D_1',D_2'\subset F'$$ such that $D_1\subset G\subset U_1$ and $D_1'\subset G\subset U_1'$ hold. If $U_1'$ is not among the $U_i$'s, then $D_1$ is contained in at least 5 sets in $\cF$ contradicting the $\vee_5$-free property of $\cF$. Now $U_1'$ contains $F,F'$ and thus $D_1,D_2,D_1',D_2'$, so unless $\{D_1,D_2\}=\{D_1',D_2'\}$ we obtain a $\wedge_5$ in $\cF$. But if $D_1=D_i'$ for some $i=1,2$, then $D_1$ is contained in $F,F'$ and $U_1,U_2,U_3$ contradicting the $\wedge_5$-free property of $\cF$.
\end{proof}

\begin{claim}\label{pupu3}
The family $\cG:=\cF\cup (\cup_{F\in \cF_1}\cN(F))$ is $\{\wedge_{41},\vee_{41}\}$-free.
\end{claim}

\begin{proof}
Suppose not, and let $D,U_1,U_2,\dots, U_{41}$ be a copy of $\vee_{41}$ in $\cG$. (The $\wedge_{41}$ case is similar.) We can assume that $D \in \cF$ as every set $G\in \cG$ contains a set (maybe itself) from $\cF$. Observe that if $D$ is contained in $G\in \cG$, then $G$ has the form of $G=D'\cup (U'\setminus F')$, where $D', F', U' \in \cF$, with $D'\subset F'\subset U'$.  In particular $D\subset U'$, and as $\cF$ is $\vee_5$-free there is at most 4 choices of $U'$. Because of the $\wedge_5$-free property, for each such $U'$, there exist at most $3\times 3=9$ pairs $F',D'$ (such that $D'\subset F'\subset U'$ and $D', F', U' \in \cF$), so the maximum number of sets from $\cG\setminus \cF$ containing the set $D \in \cF$ is $4 \times 9 = 36$. As $\cF$ is $\vee_5$-free, there are at most $4$ sets of $\cF$ containing $D$. So in total, $D$ is contained in at most $36+4 = 40$ sets from $\cG$, a contradiction. 
\end{proof}

To finish the proof observe that by Theorem \ref{Katona_Tarjan}, Claims \ref{pupu1}, \ref{pupu2}, \ref{pupu3} we have $$|\cF|+|\cF_1|\le |\cF \cup (\cup_{F\in \cF_1}\cN(F))|=|\cG|\le (1+o(1))\binom{n}{\lfloor n/2\rfloor}.$$
On the other hand, $c(P_2,\cF)$ is the number of edges (half the degree sum) in $\overrightarrow{G}_\cF$ which is at most $$\frac{1}{2}\big{[}4(|\cF|-|\cF_1|)+8|\cF_1|\big{]}=2(|\cF|+|\cF_1|)\le (2+o(1))\binom{n}{\lfloor n/2\rfloor}.$$ 
\end{proof}


\section{Proof of Theorem \ref{p4free}}
\label{p_4_free_section}

To get the upper bound, first we combine the method introduced in the previous section with some weighting of the sets. Let $\cF\subseteq 2^{[n]}$ be a $\{\wedge_k,\vee_l,P_4\}$-free family and let us define $$\cF_1:=\{F \in \cF :  (k-1)(l-1) < (l-1)d^+(F) + (k-1)d^-(F) \}.$$

\noindent 
Note that for any $F \in \cF_1$ we have $d^+(F), d^-(F) \ge 1$ and consider $U_1,...,U_{d^-(F)} \in \cF$ with $F \subseteq U_i$ for $1\le i \le d^-(F)$, and $D_1,...,D_{d^{+}(F)} \in \cF$ with $D_j \subseteq F$ for $1 \le j \le d^{+}(F)$. Let us define the following family $$\cN(F):=\{(U_i\setminus F)\cup D_j : 1\le i \le d^-(F), \ 1 \le j \le d^{+}(F)\}.$$ By definition, we have $|\cN(F)|=d^-(F)d^{+}(F)$ and $\cN(F) \cap \{U_1,\ldots,U_{d^-(F)},D_1,\ldots,D_{d^{+}(F)}\}=\emptyset$. We will prove several properties of the family $\cF \cup (\cup_{F \in \cF_1} \cN(F))$.

\begin{lemma}\label{newfree}
$\cF \cup (\cup_{F \in \cF_1} \cN(F))$ is $\wedge_{k^2l^2}$-free and $\vee_{k^2l^2}$-free.
\end{lemma}

\begin{proof}

Note that every element $G \in \cup_{F \in \cF_1} \cN(F)$ has the form of $(U \setminus F) \cup D$ with some $U,D \in \cF$ and $F \in \cF_1$ such that $D \subseteq F \subseteq U$. If it is so, then we denote $U,D$ and $F$ by $U(G), D(G)$ and $F(G)$ respectively. We also define $U(G)=D(G)=F(G)=G$ for every $G \in \cF \setminus (\cup_{F \in \cF_1} \cN(F))$. We have the following 

\begin{observation}\label{kulonbozok}

Suppose we have two different $G_1,G_2  \in \cF \cup(\cup_{F \in \cF_1} \cN(F))$ such that $U(G_1)=U(G_2)$ and $D(G_1)=D(G_2)$. Then $F(G_1) \neq F(G_2)$.

\end{observation}








Now we start the proof of Lemma \ref{newfree}. We prove that $\cF \cup (\cup_{F \in \cF_1} \cN(F))$ is $\vee_{k^2l^2}-$free, the other case is similar.  

We prove it by contradiction. Let us suppose that there are distinct sets $G, G_1,\ldots,G_{k^2l^2} \in \cF \cup (\cup_{F \in \cF_1} \cN(F))$ with $G \subsetneq G_1,\ldots,G_{k^2l^2}$. Note that we can suppose without loss of generality that $G \in \cF$ (if not, then we use $D(G) \in \cF$ instead). As we can have at most $l-1$ different sets among $U(G_1),\ldots,U(G_{k^2l^2})$ by the $\vee_l$-freeness of $\cF$, we have $k^2l$ many different $G_i$'s (we call them $G'_1,\ldots,G'_{k^2l}$) with $U(G'_1)=\ldots=U(G'_{k^2l}) \in \cF$. By the $\wedge_k$-freeness of $\cF$ there are at most $k$ different sets among $D(G'_1),\ldots,D(G'_{k^2l})$, which means we have $kl$ different sets $G'_i$ (we call them $G''_1,\ldots,G''_{kl}$) with $D(G''_1)=\ldots=D(G''_{kl}) \in \cF$. But then by Observation \ref{kulonbozok}, we have that $F(G''_1),\ldots,F(G''_{kl}) \in \cF$ are all different and all of them contain $D(G''_1)$, a contradiction and we are done with the proof of Lemma \ref{newfree}.
\end{proof}

In the following lemma, we would like to bound the quantity $|(\cup_{F \in \cF_1} \cN(F)) \setminus \cF|$ from below.  Note that it is possible that $\cN(F) \cap \cN(F') \neq \emptyset$ for two distinct sets $F,F' \in \cF$. However, we prove that it can not happen too many times in some ``average sense". Let us define an auxiliary bipartite graph $G_1$, where the two parts are $\cF_1$ and $(\cup_{F \in \cF_1} \cN(F)) \setminus \cF$, and the sets $F, H$ with $F\in \cF_1$, $H \in (\cup_{F \in \cF_1} \cN(F)) \setminus \cF$, are connected by an edge if $H \in \cN(F)$.

\begin{lemma}\label{morenew}

$$|\cF_1|\le |(\cup_{F \in \cF_1} \cN(F)) \setminus \cF|.$$

\end{lemma}

\begin{proof}

We use the following simple claim which can be described as an ``average version of Hall's theorem". We include its proof below for completeness. For a vertex $x$ in a graph $G$, let $N(x)$ denote the neighborhood of $x$ in $G$. For a set $S$ of vertices of $G$, let $N(S) = \cup_{x \in S} N(x)$.

\begin{claim}\label{averagehall}
Let $G=(A,B,E)$ be a bipartite graph such that: 

\smallskip

1. there is no isolated vertex in $B$, and 

\smallskip

2. the degree of every vertex $x\in B$ is at least the average of the degrees in $N(x)$, i.e. $$d(x)\ge \frac{\sum_{y\in N(x)}d(y)}{d(x)}.$$ 

\smallskip 

Then there exists a matching in $G$ that covers $B$, and in particular $|B|\le |A|$ holds.
\end{claim}

\begin{proof}
\[
\frac{d(x)}{\sum_{y\in N(x)}\frac{1}{d(y)}}\le \frac{\sum_{y\in N(x)}d(y)}{d(x)}\le d(x),
\]
and thus $$\sum_{y\in N(x)}\frac{1}{d(y)}\ge 1.$$
Now for any $B'\subseteq B$ we sum $\frac{1}{d(y)}$ over all edges $(xy)$ with $x\in B', y\in N(B')$, and we obtain
\[
|B'|=\sum_{x\in B'}1\le \sum_{x\in B'}\sum_{y\in N(x)}\frac{1}{d(y)}\le \sum_{y\in N(B')}\sum_{x\in N(y)}\frac{1}{d(y)}=\sum_{y\in N(B')}1=|N(B')|.
\]
As $G$ satisfies Hall's condition, $G$ indeed contains a matching that covers $B$.\end{proof}

The following claims show that the conditions of Claim \ref{averagehall} are satisfied for $G_1$.

\begin{claim} 
Condition 1. of Claim \ref{averagehall} holds for $G_1$ with $A:=(\cup_{F \in \cF_1} \cN(F)) \setminus \cF$ and $B:=\cF_1$.
\end{claim}

\begin{proof} Pick $F\in\cF_1 = B$. We want to show that $F$ is adjacent to some set in $A:=(\cup_{F \in \cF_1} \cN(F)) \setminus \cF$. By the definition of $\cF_1$, we know that $d^+(F), d^-(F) \ge 1$, and we cannot have both $d^+(F) < (k-1)/2$ and $d^-(F) < (l-1)/2$. Without loss of generality, we can assume that $d^+(F) \ge (k-1)/2$ and let $D_1,...,D_{d^{+}(F)} \in \cF$ be the sets contained in $F$. Then take a set $U_1 \in \cF$ containing $F$, and consider the sets $\{(U_1\setminus F)\cup D_j :  1 \le j \le d^{+}(F)\}$. Suppose they are all in $\cF$. Then since they are contained in $U_1$ and are different from the sets $D_1,\ldots,D_{d^{+}(F)}, F$, we get that $U_1$ contains at least $2d^{+}(F)+1 \ge k$ sets from $\cF$, contradicting the $\Lambda_k$-free property of $\cF$. Thus one of the sets $S \in \{(U_1\setminus F)\cup D_j :  1 \le j \le d^{+}(F)\}$ is not in $\cF$, so $S \in (\cup_{F \in \cF_1} \cN(F)) \setminus \cF = A$ and $S$ is adjacent to $F$ in the graph $G_1$, as desired.
\end{proof}

Now we prove Condition 2. of Claim \ref{averagehall} and we note that we will use the $P_4$-freeness of $\cF$ only during the proof of the following claim.

\begin{claim}\label{hallforus}

Condition 2. of Claim \ref{averagehall} holds for $G_1$ with $A:=(\cup_{F \in \cF_1} \cN(F)) \setminus \cF$ and $B:=\cF_1$.

\end{claim}

\begin{proof}

Pick any $F\in \cF_1$ and let $$a_i:=|\{j: (U_i \setminus F) \cup D_j \in \cF\}|, \textrm{ and}$$ $$b_j:=|\{i: (U_i \setminus F) \cup D_j \in \cF\}|,$$

where $1 \le i \le d^-(F)$ and $1 \le j \le d^+(F)$. We know that $\sum_{i=1}^{d^-(F)} a_i = \sum_{j=1}^{d^+(F)} b_j$, and let us denote this quantity by $X=X_F$. Observe that the degree of $F \in \cF_1$ in the auxiliary bipartite graph $G_1$ is $d^{-}(F)d^{+}(F) - X$.

\

Pick the set $(U_i \setminus F) \cup D_j$ for some $1 \le i \le d^-(F)$ and $1 \le j \le d^{+}(F)$, and let us examine how many sets of $\cF$ can be contained in this set. Note that by Observation \ref{kulonbozok}  the degree of $(U_i \setminus F)\cup D_j$ is at most $|\{ S \in \cF : S \subset (U_i \setminus F) \cup D_j \}|\cdot |\{ S \in \cF : S \supset (U_i \setminus F) \cup D_j \}|$.

Observe that as $(U_i \setminus F) \cup D_j \subset U_i$ we have $$|\{ S \in \cF : S \subset (U_i \setminus F) \cup D_j \}| \le k-1-(d^+(F)+1)-a_i+1=k-1-d^+(F)-a_i.$$

Indeed, the sets $F, D_1,\ldots,D_{d^{+}(F)}$ are different from the $a_i$ sets of $\cF$ of the form $(U_i \setminus F) \cup D_{j'}$, and they are all contained in $U_i$. We also claim that these sets (apart from $D_j$) are not contained in $(U_1\setminus F)\cup F_j$. This is because $\cF$ is $P_4$-free, so in particular the $D_{j'}$'s form an antichain. 
Since at most $k-1$ sets can be contained in $U_i$, it follows that besides $D_j$, at most $(k-1)-(d^{+}(F)+1)-a_i$ other sets of $\cF$ can be contained in  $(U_i \setminus F) \cup D_j$.

Similarly, we have 
$$|\{ S \in \cF : (U_i \setminus F) \cup D_j \subset S \}| \le l-1-(d^-(F)+1)-b_j+1=l-1-d^-(F)-b_j.$$

It suffices to prove the following:

\begin{equation}\label{elso}
\sum_{1 \le i \le d^-(F), \ 1 \le j \le d^{+}(F)}(k-1-d^+(F)-a_i)(l-1-d^-(F)-b_j) \le (d^{-}(F)d^{+}(F) - X)^2.
\end{equation}
Note that the left hand side of \eqref{elso} is 
$$\left(\sum_{1 \le i \le d^-(F)}(k-1-d^+(F)-a_i)\right)\left(\sum_{1 \le j \le d^{+}(F)}(l-1-d^-(F)-b_j)\right)=$$
$$=\left(d^-(F)(k-1-d^+(F))-X\right)\left(d^+(F)(l-1-d^-(F))-X\right),$$
so the desired inequality has the following form: 
\begin{equation}\label{masodik}
(d^+(F)(l-1-d^-(F))-X)(d^-(F)(k-1-d^+(F))-X) \le (d^{-}(F)d^{+}(F) - X)^2.
\end{equation}
After rearranging, \eqref{masodik} is equivalent to the following:

$$
X(4d^{+}(F)d^{-}(F)-(d^{-}(F)(k-1)+d^{+}(F)(l-1)))$$
\begin{equation}\label{harmadik}
\le d^{-}(F)d^{+}(F)(d^{-}(F)(k-1)+d^{+}(F)(l-1)-(l-1)(k-1)).
\end{equation}

Now we use the following inequalities, that are consequences of using the $\wedge_k$-freeness condition on $D_j$'s and $\vee_l$-freeness on $U_i$'s:

\bigskip  

$\bullet_1$ $X \le d^{-}(F)(k-1-d^{+}(F)),$ and $X \le d^{+}(F)(l-1-d^{-}(F)).$

\medskip 

$\bullet_2$ As a consequence, we have $X\le \frac{d^{-}(F)(k-1-d^{+}(F))+d^{+}(F)(l-1-d^{-}(F))}{2}$.

\bigskip  

Plugging $\bullet_2$ into the left hand side of \eqref{harmadik} it would be enough to prove:

$$\frac{d^{-}(F)(k-1-d^{+}(F))+d^{+}(F)(l-1-d^{-}(F))}{2}(4d^{+}(F)d^{-}(F)-(d^{-}(F)(k-1)+d^{+}(F)(l-1)))$$
\begin{equation}\label{negyedik}
\le d^{-}(F)d^{+}(F)(d^{-}(F)(k-1)+d^{+}(F)(l-1)-(l-1)(k-1)).
\end{equation}

\ 

\noindent 
If one multiplies \eqref{negyedik} by $\frac{2}{(d^{-}(F)d^{+}(F))^2}$ and uses the notation $\alpha:=\frac{k-1}{d^{+}(F)}$ and $\beta:=\frac{l-1}{d^{-}(F)}$, then \eqref{negyedik} becomes
$$(\alpha + \beta -2)(4- \alpha - \beta) \le 2 (\alpha + \beta -\alpha\beta),$$ which is equivalent to
$$0 \le (\alpha -2)^2 + (\beta -2)^2,$$
and thus we are done with the proof of Claim \ref{hallforus}.
\end{proof}

\noindent 
This finishes the proof of Lemma \ref{morenew}.
\end{proof}

\noindent 
Now observe that by Lemma \ref{newfree} and Theorem \ref{Katona_Tarjan}, we have 

\begin{equation}
\label{10}
|\cF \cup (\cup_{F \in \cF_1} \cN(F))| \le (1 + o(1))\binom{n}{\floor{n/2}}.
\end{equation}
\noindent 
Moreover, by Lemma \ref{morenew}, we have
\begin{equation}
\label{11}
|\cF \setminus \cF_1| + 2|\cF_1| = |\cF| + |\cF_1| \le |\cF|+ |(\cup_{F \in \cF_1} \cN(F)) \setminus \cF| = |\cF \cup (\cup_{F \in \cF_1} \cN(F))|.
\end{equation}
\noindent 
Combining \eqref{10} and \eqref{11}, we get 
$$ |\cF \setminus \cF_1| + 2|\cF_1| \le (1 + o(1))\binom{n}{\floor{n/2}}. $$
\noindent 
Using this, we obtain the following upper bound on the number of containments in $\cF$:

$$c(P_2, \cF)  (k+l-2)= \sum_{F \in \cF} ((l-1)d^{+}(F)+(k-1)d^-(F)) \le$$ $$ \le (k-1)(l-1)|\cF \setminus \cF_1| + 2(k-1)(l-1)|\cF_1| \le ((k-1)(l-1)+o(1))\binom{n}{n/2}.$$

This completes the proof of Theorem \ref{p4free}.
\qed



\end{document}